\documentclass[lettersize,10pt]{amsart}
\usepackage{amssymb,amscd, epsfig, mathrsfs, xypic, amsmath,amsthm, dsfont}
\usepackage{amsthm}
\xyoption{all} 
\numberwithin{equation}{section}

\theoremstyle{plain}
\newtheorem{theorem}{Theorem}[section]
\newtheorem{proposition}[theorem]{Proposition}         
\newtheorem{corollary}[theorem]{Corollary} 
\newtheorem{lemma}[theorem]{Lemma} 

\theoremstyle{definition}  
\newtheorem{definition}[theorem]{Definition}
\newtheorem{example}[theorem]{Example} 
\newtheorem{remark}[theorem]{Remark}

\newcommand{\sfG}{\mathsf G}
\newcommand{\sfT}{\mathsf T}

\newcommand{\sfK}{\mathsf K}
\newcommand{\C}{\mathbb C}   
\newcommand{\R}{\mathbb R}
\newcommand{\Q}{\mathbb Q}
\newcommand{\Z}{\mathbb Z}

\renewcommand{\P}{\mathbb P}

\DeclareMathOperator{\Lie}{Lie}
\newcommand{\w}{\mathrm{w}}
\newcommand{\twt}{a}
\newcommand{\wc}{\w \tilde{c}}
\newcommand{\trc}{\tilde{c}}
\newcommand{\wR}{{T_{\mathrm{w}}}}
\newcommand{\wRR}{{T_{\mathrm{w}}^{\R}}}
\newcommand{\trR}{T}
\newcommand{\trRR}{T^{\R}}
\newcommand{\af}{\mathrm{a}}
\newcommand{\wGr}{\mathrm{wGr}(d,n)}   
\newcommand{\Gr}{\mathrm{Gr}(d,n)}   
\newcommand{\aGr}{\mathrm{aPl}(d,n)}
\newcommand{\aGrt}{\mathrm{aPl}(d,n)^{\times}}
\newcommand{\Ctimes}{\mathbb{C}^{\times}}

\newcommand{\bmH}{\overline{H}}

\newcommand{\di}{\mathrm{div}}
\newcommand{\id}{\mathrm{id}}
\newcommand{\tS}{\tilde{S}}
\newcommand{\aS}{\mathrm{a}\tilde{S}}
\newcommand{\wS}{\w\tilde{S}}

\newcommand{\str}{\left\{\!\begin{smallmatrix}n\\d\end{smallmatrix}\!\right\}}
\newcommand{\strEG}{\left\{\!\begin{smallmatrix}4\\2\end{smallmatrix}\!\right\}}

\newcommand{\wD}{\w\!D}
\newcommand{\wDR}{\w\!D^{\R}}

\newcommand{\wu}{\w u}
\newcommand{\tow}{y^w}

\newcommand{\KC}{K}
\newcommand{\KR}{K^{\R}}
\newcommand{\hmap}{\pi}
\newcommand{\whmap}{\pi_{\w}}
\newcommand{\kmap}{\kappa}
\newcommand{\wkmap}{\kappa_{\w}}
\newcommand{\wy}{y^{\mathrm{w}}}
\newcommand{\zz}{z}

\makeatletter

\newcommand{\Rmnum}[1]
{\expandafter\@slowromancap\romannumeral #1@}
\makeatother

\DeclareMathOperator{\invv}{inv}

\begin{document}     

\title[]
{Equivariant cohomology of weighted Grassmannians \\ and weighted Schubert classes}

\author{Hiraku Abe and Tomoo Matsumura}

\date{\today}

\maketitle

\begin{abstract}
In this paper, we study the $\wR$-equivariant cohomology of the weighted Grassmannians $\wGr$ introduced by Corti-Reid \cite{CR} where $\wR$ is the $n$-dimensional torus that naturally acts on $\wGr$. We introduce the equivariant \emph{weighted Schubert classes} and, after we show that they form a basis of the equivariant cohomology, we give an explicit formula for the structure constants with respect to this Schubert basis. We also find a linearly independent subset $\{\wu_1,\cdots, \wu_n\}$ of $\Lie(\wR)^*$ such that those structure constants are polynomials in $\wu_i$'s with non-negative coefficients, up to a permutation on the weights.
\end{abstract}


\section{Introduction} \label{section-intro}

The \emph{weighted Grassmannian} $\wGr$ introduced and studied by Corti-Reid \cite{CR}, following the work of Grojnowski, is a projective variety with at worst orbifold singularity with a torus action. It is a generalization of the ordinary Grassmannian and is defined in a weighted projective space by the well-known \emph{Pl\"{u}cker relations} as weighed homogeneous polynomials with appropriate weights. In this paper, we define the \emph{weighted Schubert classes} and show that they will form a basis of the equivariant/non-equivariant cohomology of $\wGr$ over $\Q$-coefficients. Our main goal is to study the structure constants of the cohomology rings with respect to these weighted Schubert classes. The explicit formula of these structure constants for the weighted Grassmannian is derived from any formula for the ordinary Grassmannian (for example, the Knutson-Tao's puzzle formula \cite{KT}), by detouring to the equivariant cohomology of the quasi-projective variety $\aGrt$ defined in the affine space by the Pl\"{u}cker relations. We have found appropriate equivariant parameters in which the equivariant structure constants are polynomials with non-negative rational coefficients when the weights are non-decreasing (hence this implies that the structure constants of the ordinary cohomology are also non-negative). This is an analogue of the equivariant positivity proved by Graham \cite{GrahamPos}, although we do not have the geometric or representation-theoretic interpretation of those parameters, while as, in \cite{GrahamPos}, they are the simple roots in the character group of the maximal torus when we regard the flag varieties as homogeneous varieties. 

Below we summarize our results in detail. Recall that the ordinary Grassmannian $\Gr$ is the space of $d$-dimensional subspaces in the $n$-dimensional complex plane $\C^n$. It can be described as a non-singular projective variety of dimension $d(n-d)$ defined by the well-kwown homogeneous polynomials, called the \emph{Pl\"ucker relations}. It is embedded in the projective space $\P(\C^{\str})$ where $\str:=\{\{\lambda_1,\cdots,\lambda_d\} \ |\ 1\leq \lambda_1 < \cdots < \lambda_d \leq n\}$ and $\C^{\str}$ is the affine space of the pl\"{u}cker coordinates. Let $\aGr$ be the affine variety in $\C^{\str}$ defined by the pl\"{u}cker relation and let $\aGrt:=\aGr - \{0\}$. The $(n+1)$-dimensional complex torus $\KC:=(\C^{\times})^n \times \C^{\times} = \trR \times \C^{\times}$ naturally acts on $\aGr$ and $\aGrt$, through the homomorphism
\[
\rho : \KC \to (\C^{\times})^{\str}, \ \ (t_1,\cdots, t_n,s) \mapsto \left( \ st_{\lambda}\ \right)_{\lambda \in \str}, \ \ \mbox{where}\  t_{\lambda}:=t_{\lambda_1} \cdots t_{\lambda_d}.
\]
The Grassmannian $\Gr$ is the quotient of $\aGrt$ by the $\C^{\times}$-action of the last component of $\KC$. Following \cite{CR}, we define the \emph{weighted Grassmannian} $\wGr$ as the quotient of $\aGrt$ by the \emph{locally free} action of a ``twisted diagonal" $\wD$ in $\KC$: for $w:=(w_1,\cdots, w_n) \in (\Z_{\geq 0})^n$ and $a \in \Z_{\geq 1}$, let 
\begin{align*}
\wD := \{(t^{w_1},\dots,t^{w_n}, t^a) \in \KC \mid t\in \C^{\times} \}.
\end{align*}
The weighted Grassmannian is defined as 
\[
\wGr:= \aGrt/\wD,
\] 
together with the residual action of the quotient torus $\wR :=\KC/\wD$. It is a projective variety with at worst orbifold singularities, naturally embedded in the weighted projective space $\P_w(\C^{\str}):=( \C^{\str}-\{0\})/ \wD$ with the weights
\[
\left( \ w_{\lambda}:=w_{\lambda_1} + \cdots + w_{\lambda_d} + a \ \right)_{\lambda \in \str}.
\]

In Section \ref{section-basics}, we study the analogue of the usual \emph{Schubert cell (Bruhat) decomposition} for $\wGr$ (Proposition \ref{quasi-cell decomp}) and then by the standard argument we show our first result (more precise versions of the claims will be in the main body of the paper). \emph{In this paper, all cohomologies are assumed to be the singular cohomologies over $\Q$-coefficients unless otherwise specified.} 

\vspace{0.1in}
\noindent {\bf Proposition A} (Proposition \ref{vanishing of odd}, \ref{prop action 500}). The cohomology $H^*(\wGr)$ is concentrated in even degree. As a consequence, the equivariant cohomology $H_{\wR}^*(\wGr)$ is a free module over $H^*(B\wR)$.

\vspace{0.1in}
In Section \ref{section-picture}, we explain the following key isomorphisms among the equivariant cohomology rings of $\Gr$, $\aGrt$ and $\wGr$. The claim follows essentially from the \emph{Vietoris-Begle mapping theorem}. 

\vspace{0.1in}
\noindent {\bf Proposition B} (Proposition \ref{key prop}). The pullback maps on the equivariant cohomologies
\[
\xymatrix{
H_{\trR}^*(\Gr) \ar[r]^{{\hmap}^*} & H_{\KC}^*(\aGrt)  &H_{\wR}^*(\wGr) \ar[l]_{\ \ \ {\whmap}^*}
}
\]
are isomorphisms of rings over $H^*(B\trR)$ and $H^*(B\wR)$ respectively.
\vspace{0.1in}

\noindent Having these isomorphisms, we introduce the equivariant Schubert classes $\aS_{\lambda}$ of $\aGrt$ to be the image of the usual equivariant Schubert class $\tS_{\lambda}$ in $H_{\trR}^*(\Gr)$ under the pullback $\pi^*$ and  define the \emph{equivariant weighted Schubert classes} $\wS_{\lambda}$ of $\wGr$ by
\[
\wS_{\lambda}:=(\whmap^*)^{-1}(\aS_{\lambda}).
\]
The corresponding ordinary cohomology classes for $\Gr$ and $\wGr$ are denoted by $S_{\lambda}$ and $\w S_{\lambda}$ respectively.

\vspace{0.1in}
In Section \ref{section-GKM}, we obtain the \textit{GKM (Goresky-Kottwitz-Macpherson)} descriptions of $H_{\KC}^*(\aGrt)$ and $H^*_{\wR}(\wGr)$, following \cite{CS}, \cite{GKM} and \cite{HM}. We observe that there is the following commutative diagram of injective localization maps
\begin{equation*}
\xymatrix{
H_{\trR}^*(\Gr) \ar[rr]\ar[d]_{\hmap^*}^{\cong}            && \bigoplus_{\lambda} H^*(B\trR)\ar[d]_{\cong}\\
H_T^*(\aGrt) \ar[rr]                                    && \bigoplus_{\lambda} H^*(B\KC_{\lambda})\\
H_{\wR}^*(\wGr) \ar[rr]\ar[u]^{\whmap^*}_{\cong} && \bigoplus_{\lambda} H^*(B\wR) \ar[u]^{\cong}
}
\end{equation*}
where $\KC_{\lambda}$ is the kernel of $\KC \to \C^{\times};\  t \mapsto t_{\lambda}$. The GKM descriptions of $H_{\KC}^*(\aGrt)$ and $H_{\wR}^*(\wGr)$ are obtained in Proposition \ref{GKM} and \ref{aGr GKM}, from the well-known one for $H_{\trR}^*(\Gr)$ by the commutative diagram. Furthermore, the upper triangularity of the image of $\aS_{\lambda}$ and $\wS_{\lambda}$ is given in Proposition \ref{upper triangularity aS} and \ref{can prop 50}, and as a consequence, we have

\vspace{0.1in}
\noindent{\bf Proposition C} (Proposition \ref{can prop 120}). \!$\{\wS_{\lambda}\}_{\lambda}$ is a basis of $H_{\wR}^*(\wGr)$ as a module over $H^*(B\wR)$.

\vspace{0.1in}
This allows us to define the structure constants $\wc_{\lambda\mu}^{\nu}$ of $H_{\wR}^*(\wGr)$ by
\[
\wS_{\lambda} \cdot \wS_{\mu} = \sum_{\nu} \wc_{\lambda\mu}^{\nu} \wS_{\nu} \ \ \ \mbox{where}\  \wc_{\lambda\mu}^{\nu} \in H^*(B\wR).
\]
In Section \ref{section-final}, we derive the formula for $\wc_{\lambda\mu}^{\nu}$ and prove the equivariant positivity. Let $\{y_1,\dots, y_n, z\}$ be the standard basis of $\Lie(\KC)_{\Z}^*$ and identify $H^*(B\KC)$ with $\Q[y_1,\cdots, y_n,z]$. Since $\wR$ is a quotient of $\KC$, we can regard $\Lie(\wR)_{\Z}^*$ as a subspace of $\Lie(\KC)_{\Z}^*$.  For each pair $\alpha=(i,j)$ of integers in $[n]$ with $i>j$, let
\[
 u_{\alpha}:=y_i-y_j \in \Q[\trR^*] \ \ \ \  \ \ \mbox{ and }\ \ \ \ \ \ \wu_{\alpha}:=(\wy_i-\wy_j)-\frac{w_i-w_j}{w_{\id}}\wy_{\id} \in \Q[\wR^{\!\!*}]
\]
where $\id \in\str$ is the unique minimum element in the Bruhat order and $y_{\lambda}:=y_{\lambda_1}+\cdots + y_{\lambda_d}$.  It is easy to see that $\{\wu_{(i+1,i)}, i=1,\cdots,n-1\}$ is a linearly independent subset of $\Lie(\wR)_{\Q}^*$. For simplicity, we let $u_i:=u_{(i+1,i)}$ and $\wu_i:=\wu_{(i+1,i)}$. For each finite collection $I=\{\alpha_1,\cdots,\alpha_p\}$ of pairs of integers in $[n]$ as above, let
\begin{align*}
u_I:=u_{\alpha_1}\cdots u_{\alpha_p} \ \ \ \mbox{and} \ \ \  \wu_{I}^{(r)}
=\sum_{1\leq s_1<\cdots<s_r\leq p}  \frac{w(\alpha_{s_1})}{w_{\id}}\cdots \frac{w(\alpha_{s_r})}{w_{\id}} \frac{\wu_{\alpha_1}\cdots\wu_{\alpha_p}}{\wu_{\alpha_{s_1}}\cdots\wu_{\alpha_{s_r}}}.
\end{align*}
where  $w(\alpha):=w_i-w_j  \in \Z$ if $\alpha=(i,j)$. Here note that $\wu_{I}^{(0)}=\wu_{\alpha_1}\cdots\wu_{\alpha_p}$. 


We first introduce $K_{1^r \eta}^{\nu} \in \Q[\wR^{\!\!*}]$ as the coefficient for the following product.
\begin{align*}
(\aS_{\di})^{r} \aS_{\eta}
= \sum_{\nu} K_{1^r \eta}^{\nu} \aS_{\nu}.
\end{align*}
The explicit formula for $K_{1^r \eta}^{\nu}$ is given in Lemma \ref{lem:K1eta}.  To obtain the formula for $\wc_{\lambda\mu}^{\nu}$, we use the well-known fact that the equivariant Schubert structure constant $\tilde{c}_{\lambda\mu}^{\nu}$ for $H_{\trR}^*(\Gr)$ is an element of $\Z[u_1,\cdots,u_{n-1}]$ 
\begin{align}
 \tilde{c}_{\lambda\mu}^{\nu}=\sum_{|I| = l(\lambda) + l(\mu) -l(\nu) } c(\lambda,\mu,\nu;I) u_{I}, \ \  c(\lambda,\mu,\nu;I) \in \Z_{\geq 0}
\end{align}
where $I$ runs over collections of pairs $(i,j)$ of integers in $[n]$ with $i>j$ as above. For example, Knutson-Tao \cite{KT} computed the number $c(\lambda,\mu,\nu;I)$ in terms of the equivariant puzzles.

The following is our main theorem.

\vspace{0.1in}
\noindent{\bf Theorem E} (Theorem \ref{main theorem}, \ref{equivariant positivity}).\ 
Let $\lambda,\mu,\nu\in\str $, then
\begin{align}
\wc_{\lambda\mu}^{\nu}=\sum_{\nu\geq\eta\geq\lambda,\mu}\sum_{I}\sum_{r=0}^{|I|} c(\lambda,\mu,\eta;I) K_{1^r \eta}^{\nu} \wu_{I}^{(r)}.
\end{align}
Moreover, if $w_1\leq w_2\leq\cdots\leq w_n$, then $\wc_{\lambda\mu}^{\nu}$ is a polynomial in $\wu_1,\cdots,\wu_{n-1}$ with non-negative coefficients. 
\vspace{0.1in}

It is worth noting that a permutation of the index $(1, \cdots, n)$ can change the order of  the weights $\{w_1,\cdots,w_n\}$ into a non-decreasing order without changing the space $\wGr$ up to isomorphisms. Therefore we can always find a Schubert basis that satisfies the positivity.

The ordinary structure constants $\w c_{\lambda\mu}^{\nu}$ are given by the non-equivariant limit $\wu_1= \cdots = \wu_{n-1} = 0$. Thus we obtain the formula for $\w c_{\lambda\mu}^{\nu}$ too, particularly in terms of a specielization of the equivariant structure constants $\tilde{c}_{\lambda\mu}^{\nu}$ for $\Gr$ computed in \cite{KT}:
 
\vspace{0.1in}
\noindent{\bf Corollary G} (Corollary \ref{cor of main} below) 
Let $\lambda,\mu,\nu\in\str$.  The structure constant $\w c_{\lambda\mu}^{\nu}$ is given by
\begin{align*}
\w c_{\lambda\mu}^{\nu}  = \sum_{\nu\geq\eta\geq\lambda,\mu}\sum_{\substack{\nu = \nu^0 \rightarrow \\ \cdots\rightarrow \nu^l = \eta}}\frac{ \tilde{c}_{\lambda\mu}^{\eta}(u_i=w_{i+1}-w_i, i=1,\cdots,n-1)  }{w_{\nu^1} \cdots w_{\nu^l}}
\end{align*}
if $l(\lambda)+l(\mu) =l(\nu)$ and $\w c_{\lambda\mu}^{\nu}=0$ if otherwise. Furthermore, if $w_1\leq w_2\leq \cdots\leq w_n$, $ \w c_{\lambda\mu}^{\nu}$ is non-negative. 
\vspace{0.1in}


We conclude Section \ref{section-final} by including the examples of $\mathrm{wGr}(1,n)$ and $\mathrm{wGr}(2,4)$. The space $\mathrm{wGr}(1,n)$ is the well-known \emph{weighted Projective space} and its integral cohomology are first studied by Kawasaki \cite{Kawasaki} and its equivariant cohomology by Bahri-Franz-Ray\cite{BFR} and Tymoczko \cite{Tym}. We discuss the relation of our Schubert basis to Kawasaki's basis over $\Z$-coefficients at Example \ref{kawasakibasis}. 


\section{Weighted Grassmannians and Schubert cell decomposition}\label{section-basics}
In this section, we recall the definition of the \emph{weighted Grassmannian} $\wGr$, following \cite{CR}.  We study the coordinate charts and obtain a \emph{quasi-cell decomposition} which generalizes the usual Schubert cell decomposition of the ordinary Grassmannian $\Gr$. This allows us to show that  the odd degree classes of the rational cohomology of $\wGr$ vanish and also the equivariant cohomology is a free module over a polynomial ring.

For positive integers $d$ and $n$ such that $d <n$, let $[n]:=\{1,\cdots, n\}$, and 
\[
\str :=\{\lambda \subset [n]\ | \ |\lambda|=d\}.
\]
We denote the elements of $\lambda$ by $\lambda_1,\cdots, \lambda_d$ where $\lambda_1 < \cdots < \lambda_d$. For $\lambda,\mu\in\str$, we define the \textit{Bruhat order} by
\begin{equation}\label{Bruhatorder}
 \text{$\lambda\geq\mu$ \ \ \ if\ \ \  $\lambda_i\leq \mu_i$ for all $i=1,\cdots,d$}.
\end{equation}
An \textit{inversion} $(k,l)$ of $\lambda$ is a pair of $k \in \lambda$ and $l\not\in\lambda$ such that $k< l$. Let $\invv(\lambda)$ be the set of all inversions of $\lambda$. The \emph{length} $l(\lambda)$ of $\lambda$ is defined to be the cardinarity of $\invv(\lambda)$. For each $(k,l)\in \invv(\lambda)$, let $(k,l)\lambda$ be the element of $\str$ obtained by replacing $k$ in $\lambda$ by $l$. 
We say that $\lambda$ \emph{covers} $\mu$ if $\lambda\geq\mu$ and $l(\lambda)=l(\mu)+1$, and denote $\lambda \to \mu$. 
\begin{remark}\label{dictionary for KT}
For each $\lambda\in\str$, we can consider the sequence of ones and zeros such that one for each $\lambda_i$-th position ($i=1,\cdots,d$) and zeros for the other positions. Obviously this is a bijection, and we can identify $\str$ and the set of sequences of $d$ zeros and $n-d$ ones by this rule. This will help us to compare our notation and the notations in \cite{KT}.
\end{remark}
\subsection{The weighted Grassmannian} \label{subsection-def of wGr}
Let $\C^n$ be the complex $n$-plane and $\wedge^d \mathbb{C}^n$ its $d$-th exterior product. The standard representation of $\text{GL}_{n}(\C)$ on $\C^n$ canonically induces the representation of  $\text{GL}_{n}(\C)$ on $\wedge^d \C^n$ and hence a linear $\text{GL}_{n}(\C)$-action on $\P(\wedge^d \C^n)$. Through the Pl\"{u}cker embedding, the Grassmannian $\Gr$ of $d$-dimensional subspaces in $\C^n$ can be indentified with the $\text{GL}_{n}(\C)$-orbit  $\text{GL}_{n}(\C) \cdot [e_1\wedge\cdots\wedge e_d]$ in $\P(\wedge^d \C^n)$. Consider the affine cone of $\Gr$ in $\wedge^d \C^n$
\[
\aGr := \text{GL}_{n}(\C) \cdot (e_1\wedge\cdots\wedge e_d).
\]
Let $\trR:=(\C^{\times})^n$ and identify it with the diagonal matrices of $\text{GL}_{n}(\C)$. The quasi-affine variety $\aGrt := \aGr -\{0\}$ is preserved under the action of $\KC:=\trR\times \C^{\times}$  on $\wedge^d \C^n$ where the first factor acts through the $\text{GL}_{n}(\C)$-action and the second factor by the scalar multiplication.
\begin{definition}[Corti-Reid \cite{CR}]  Let $w:=(w_1,\cdots,w_n) \in (\Z_{\geq0})^n$ and $\twt \in\Z_{\geq1}$. Consider the subgroup of $\KC$ and the corresponding quotient group:
\begin{align*}
\wD := \{(t^{w_1},\cdots,t^{w_n},t^{\twt}) \in \KC  \mid t\in \C^{\times} \} \ \ \ \ \ \ \mbox{ and } \ \ \  \ \ \ \wR:= \KC /\wD.
\end{align*}
The \textit{weighted Grassmannian} $\wGr$ is the projective variety with at worst orbifold singularities, defined by
\begin{align*}
\wGr := \aGrt/\wD.
\end{align*}
The quotient map $\whmap:\aGrt\rightarrow\wGr$ is equivariant with respect to the quotient homomorphism $\wkmap : \KC \rightarrow \wR$. 
\end{definition}
When $w=(0,\cdots,0)$ and $\twt=1$, $\wGr$ becomes the usual Grassmannian $\Gr$ and $\wR$ is canonically identified with $\trR$. In this case, we denote the above quotient maps by $\kmap : \KC\rightarrow \trR$ and $\hmap:\aGrt\rightarrow \Gr$ respectively.
\subsection{The Charts for $\aGrt $ and $\wGr $} \label{aff chart}
The standard basis $\{e_i, i \in [n]\}$ of $\C^n$ induces the standard basis $\{ e_{\lambda} := e_{\lambda_1} \wedge \cdots \wedge e_{\lambda_d}, \lambda\in\str\}$ of $\wedge^d\C^n$. We identify $\bigwedge^d \C^n$ with a coordinate space $\C^{\str}$ where a vector $x=\sum x_{\lambda} e_{\lambda}$ corresponds to the coordinate vector $x=(x_{\lambda})_{\lambda\in\str}$. For each $(t,s) \in \KC=\trR\times \C^{\times}$, its action on $e_{\lambda}$ is given by 
\begin{align}\label{definition for t_lambda}
(t,s)\cdot e_{\lambda} = (st_{\lambda}) e_{\lambda} \quad \text{where} \quad t_{\lambda}:=\textstyle{\prod_{l\in \lambda}} t_l.
\end{align}
Let $U^{\lambda}:=\left\{ [x] \in\Gr \mid x_{\lambda} \neq 0 \right\}$. It is well-known that $U^{\lambda}$ is $\trR$-equivariantly isomorphic to the complex affine space $\C^{d(n-d)}$ with a linear $\trR$-action. It is a $\trR$-invariant affine chart of $[e_{\lambda}]$, and these form an affine open cover of  $\Gr$. 

The quotient map $\hmap : \aGrt \rightarrow \Gr$ is  a $\kmap$-equivariant $\C^{\times}$-principal bundle. The preimage of $U_{\lambda}$ is
\begin{align*}
\af U^{\lambda}:=\left\{ x \in\aGrt \mid x_{\lambda} \neq 0 \right\}.
\end{align*}
and then a $\kmap$-equivariant trivialization  is given by
\begin{align}
 \psi_{\lambda}:\af U^{\lambda} \to U^{\lambda}\times\C^{\times};
 \quad  x \mapsto (\hmap(x), x_{\lambda})
\end{align}
where the $\KC$-action on $U^{\lambda}\times\C^{\times}$ is defined by $(t,s) \cdot ([x],y) := ([tx],st_{\lambda}y)$.  Indeed, the inverse map is given by
\begin{align}
 U^{\lambda}\times\C^{\times} \to \af U^{\lambda};  
 \quad  ([x],t) \mapsto (tx_{\lambda}^{-1}x_{\eta})_{\eta\in\str}.
\end{align}
This $\af U^{\lambda}$ plays a role of  a $\KC$-invariant chart of $e_{\lambda}$ in $\aGrt$.

The quotient $\w U^{\lambda}:=\af U^{\lambda}/\wD$ gives us a $\wR$-equivariant open neighborhood of $[e_{\lambda}]\in\wGr$  and $\psi_{\lambda}$ induces an equivariant isomorphism
\begin{align*}
 \overline{\psi}_{\lambda} : \w U^{\lambda}  \stackrel{\cong}{\longrightarrow} (U^{\lambda} \times \Ctimes )/\wD.
\end{align*}
Let 
\begin{align}\label{definition for w_lambda}
w_{\lambda}:= a + \sum_{i=1}^d w_{\lambda_i} \quad \text{ for each $\lambda\in\str$}.
\end{align}
Then the finite cyclic subgroup of $\wD$
\[
G_{\lambda}  =  \left\{ (t^{w_1},\cdots,t^{w_n},t^{\twt}) \in \wD \mid t\in \C^{\times} \ \mbox{ and } \ t^{w_{\lambda}}=1 \right\}.
\]
acts on the second factor of $U^{\lambda} \times \Ctimes$ trivially. Hence the image of the isomorphism $\overline{\psi}_{\lambda}$ is equivariantly homeomorphic to $U^{\lambda}/G_{\lambda}$. 
\subsection{The Schubert cell decompositions}\label{subsection-decomp}
For each $\lambda\in\str$, we have the Schubert cell $\Omega^{\circ}_{\lambda}$ and the Schubert variety $\Omega_{\lambda}$ in $\Gr$ (See \cite{Ful} or \cite{KT}). Define
\begin{align*}
 \af\Omega^{\circ}_{\lambda}:=\hmap^{-1}(\Omega^{\circ}_{\lambda})
 \quad \text{and} \quad
 \af\Omega_{\lambda} := \hmap^{-1}(\Omega_{\lambda}).
\end{align*}
Under the chart $\psi_{\lambda}$, we have $\af \Omega^{\circ}_{\lambda} \cong \Ctimes\times\Omega^{\circ}_{\lambda}$ and so its complex codimension in $\aGrt$ is the length $l(\lambda)$. The irreducibility of $\af\Omega_{\lambda}$ follows from the irreducibility of $\Omega_{\lambda}$ and the fiber $\C^{\times}$ of the bundle $\hmap$. Therefore the closure of the open subset $\af\Omega^{\circ}_{\lambda}\subset\af\Omega_{\lambda}$ coincides with $\af\Omega_{\lambda}$. The usual Schubert cells decomposition $\Gr  = \textstyle{\coprod}_{\lambda \in \str} \Omega^{\circ}_{\lambda}$ induces  the $\KC$-invariant decomposition
\begin{align}\label{q-cell 2100}
 \aGrt  = \textstyle{\coprod}_{\lambda \in \str} \af \Omega^{\circ}_{\lambda}.
\end{align}
By the $\KC $-invariancy, it descends to the quotient $\wGr$ and gives the $\wR$-invariant decomposition:
\begin{proposition}\label{quasi-cell decomp}
\[
\wGr = \textstyle{\coprod}_{\lambda \in \str} \w \Omega^{\circ}_{\lambda}\ \ \ \mbox{ where }\ \ \  \w\Omega^{\circ}_{\lambda}:= \af \Omega^{\circ}_{\lambda}/\wD.
\] 
Under the chart $\overline{\psi}_{\lambda}$, $\w \Omega^{\circ}_{\lambda} \cong \Omega^{\circ}_{\lambda}/G_{\lambda}$.
\end{proposition}
We call this decomposition a \emph{quasi-cell decomposition} because each ``cell" is homeomorphic to an Euclidean space modulo a finite group.

\subsection{Vanishing of the odd degree}\label{vanishing odd}
The argument of Appendix B in \cite{Ful} can be applied to the quasi-cell decomposition, and we obtain
\begin{align}\label{prop coh 300}
\bmH_{i}(\wGr) \cong \bigoplus_{2\dim\w\Omega^{\circ}_{\lambda}=i} \bmH_{i}(\w\Omega_{\lambda}^{\circ})
\end{align}
where $\bmH_*$ is the rational Borel-Moore homology.  We have
\[
\bmH_{i}(\w\Omega_{\lambda}^{\circ})= \bmH_i(\Omega_{\lambda}^{\circ}/G_{\lambda})  \cong H^{2\dim\Omega^{\circ}_{\lambda} -i}(\Omega_{\lambda}^{\circ}/G_{\lambda})\cong H^{2\dim\Omega^{\circ}_{\lambda} -i}(\Omega_{\lambda}^{\circ})^{G_{\lambda}}
\]
where the second equality follows from the rational Poincar$\acute{\text{e}}$ duality (c.f. \cite{CLS} Prop 13.A.4, Appendix, Chap. 13) and the third equality is well-known (see \cite{BBFMP}). Since the $G_{\lambda}$ acts on $\Omega_{\lambda}^{\circ}$ through the action of a connected group, $G_{\lambda}$ acts on $H^*(\Omega_{\lambda}^{\circ})$ trivially. Therefore after we apply the Poincar$\acute{\text{e}}$ duality again, we obtain
\begin{proposition}
\label{vanishing of odd}
\begin{align*}
 H^i(\wGr)  \cong
 \begin{cases}
  \bigoplus_{2l(\lambda)=i}\Q &\text{if $i$ is  even}, \\
  0 & \text{if $i$ is odd}.
\end{cases}
\end{align*}
\end{proposition}
Recall that the equivariant cohomology for the $\wR$-action on $\wGr$ is defined as the cohomology of the \emph{Borel construction}, i.e. the total space of the fibration
\begin{align*}
 \wGr \stackrel{\zeta}{\hookrightarrow} E\wR\times_{\wR}\wGr  \to B\wR,
\end{align*}
where $E\wR \to B\wR$ is a universal principal $\wR$-bundle with a contractible total space and $E\wR\times_{\wR}\wGr := (E\wR\times\wGr)/\wR$. The pullback along the projection makes $H_{\wR}^*(\wGr)$ an $H^*(B\wR)$-module. Since the fiber is path-connected, the vanishing of odd degree classes implies that the associated Serre spectral sequence collapses at $E_2$-stage. Thus we have
\begin{proposition} \label{prop action 500}
As $H^*(B\wR)$-modules,
\begin{align*}
H_{\wR}^*(\wGr) \cong H^*(B\wR) \otimes_{\Q} H^*(\wGr).
\end{align*}
In particular, $H_{\wR}^*(\wGr)$ is a free module over $H^*(B\wR)$.
\end{proposition}
\section{Equivariant Weighted Schubert Classes}\label{section-picture}
In this section, we observe that the rational equivariant cohomologies $H_{\KC}^*(\aGrt)$, $H_{\wR}^*(\wGr)$, and $H^*_{\trR}(\Gr)$ are all isomorphic as rings, while they are modules over difference polynomial rings. We define the \emph{equivariant weighted Schubert classes} $\wS_{\lambda}$ in $H_{\wR}^*(\wGr)$ using these ring isomorphisms.

The quotient maps from $\aGrt$ to $\wGr$ and $\Gr$, and from $\KC$ to $\wR$ and $\trR$,  induce the following commutative diagram of the Borel constructions:
\[
\xymatrix{
E\trR\times_{\trR} \Gr\ar[d] && E\KC\times_\KC \aGrt \ar[ll]_{\hmap}\ar[rr]^{\whmap} \ar[d] && E\wR \times_{\wR} \wGr\ar[d] \\
B\trR && B\KC\ar[rr]^{\wkmap}\ar[ll]_{\kmap} && B\wR
}
\]
By the functoriality, the pullback maps on cohomologies
\[
\hmap^*: H_{\trR}^*(\Gr) \to H_{\KC}^*(\aGrt) \ \ \ \mbox{ and } \ \ \ \whmap^*:H_{\wR}^*(\wGr) \to H_{\KC}^*(\aGrt)
\] 
are homomorphism of rings over the polynomial rings $H^*(B\trR)$ and $H^*(B\wR)$ respectively. The proof of the following proposition is postponed until after we define the weighted Schubert classes.
\begin{proposition}
\label{key prop}
The maps $\hmap^*$ and $\emph{$\whmap^*$}$ are isomorphisms as rings over the poylnomial rings $\emph{$H^*(B\trR)$}$ and $\emph{$H^*(B\wR)$}$ respectively.
\end{proposition}
Since each $\af \Omega_{\lambda}$ is a closed $\KC$-invariant irreducible subvariety in a non-singular quasi-projective $\KC$-variety $\aGrt$, the equivariant Gysin map $H_{\KC}^*(\af \Omega_{\lambda})\rightarrow H_{\KC}^{*+2l(\lambda)}(\aGrt)$ (c.f. \cite[Appendix B]{Ful}) defines the cohomology class $[\af \Omega_{\lambda}]_{\KC}$ in  $H_{\KC}^*(\aGrt)$ associated to $\af \Omega_{\lambda}$ as the image of $1\in H_{\KC}^0(\af \Omega_{\lambda})$:
\begin{align*}
\aS_{\lambda}:=[\af \Omega_{\lambda}]_{\KC} 
\in H_{\KC}^{2l(\lambda)}(\aGrt).
\end{align*}
Since $\af\Omega_{\lambda}=\hmap^{-1}(\Omega_{\lambda})$ and $\hmap:\aGrt\rightarrow\Gr$ is an equivariant fiber bundle with respect to the quotient $\kmap:\KC\rightarrow \trR$, we actually have
\begin{align*}
\aS_{\lambda}=\hmap^*(\tS_{\lambda}) 
\end{align*}
where $\tS_{\lambda}=[\Omega_{\lambda}]_{\trR}$ is the $\trR$-equivariant Schubert class in $H_{\trR}^*(\Gr)$ (c.f. Appendix B-(8) in \cite{Ful}).
\begin{definition}
Define the \emph{$\emph{$\wR$}$-equivariant weighted Schubert class} corresponding to $\lambda$ by
\begin{align*}
 \wS_{\lambda}:=(\whmap^*)^{-1}(\aS_{\lambda}) \in H_{\wR}^{2l(\lambda)}(\wGr).
\end{align*}
This definition coincides with the usual equivariant Schubert class for the ordinary Grassmannian when the weights are trivial. We denote by $S_{\lambda}$ and $\w S_{\lambda}$ the corresponding classes in \textit{ordinary} cohomologies $H^*(\Gr)$ and $H^*(\wGr)$ respectively.
\end{definition}
\subsection*{Proof of Proposition \ref{key prop}} \label{subsection-Matsumura map}
By an elementary application of the \emph{Vietoris-Begle mapping theorem} (c.f. \cite[Thm.15, Sec.9, Chp.6]{Spa}), we have the following lemma.
\begin{lemma}
\label{key lemma}
Let $M$ be a compact manifold with a smooth action of a compact torus $\sfK$. Let $\sfG \subset \sfK$ be a subtorus that acts on $M$ with finite stabilizers. Let $\sfT := \sfK/\sfG$. Then the natural map $\theta: E\sfK\times_{\sfK} M \to E\sfT\times_{\sfT} (M/\sfG)$ induces an isomorphism of rings over $H^*(B{\sfT})$ on the rational equivariant cohomology
\[
\theta^*: H_{\sfT}^*(M/\sfG) \to H_{\sfK}^*(M).
\]
\end{lemma}
\noindent
Thus, we need to prepare only a description of $\wGr$ as the quotient of a compact space by a real torus. Let $\KR, \wDR, \trRR$ and $\wRR$ be the real tori in $\KC, \wD, \trR$ and $\wR$ respectively. Recall that we have a natural isomorphism $H_{\KC}^*(Y)\cong H_{\KR}^*(Y)$ for any $\KC$-space $Y$. Since the $\wDR$-action on $\C^{\str}$ factors through the canonical $(S^1)^{\str}$-action, it is hamiltonian with the standard moment map. Since $\aGrt$ is a $\wDR$-invariant symplectic submanifold of $\C^{\str}$, there is the induced moment map \footnote{Here we identify $\Lie(\wD) \cong \Lie(\C^{\times}) =\R$ by the map $S^1 \to \wD (t \mapsto (t^{dw_1 +a}, \cdots, t^{dw_n +a}))$.}
\begin{align*}
 \Psi : \aGrt \to \R \quad ; \quad x \mapsto -\frac{1}{2}\sum_{\lambda \in \str} d\cdot w_{\lambda}|x_{\lambda}|^2
\end{align*}
where the integer $w_{\lambda}$ is defined at (\ref{definition for w_lambda}). For a  regular value $\xi$, the preimage $M:=\Psi^{-1}(\xi)$ is a compact $\KR$-invariant submanifold of $\aGrt$. Moreover there is a $\KR$-equivariant deformation retraction from $\aGrt$ to $M$ given by the homotopy
\[
F : \aGrt \times I \to \aGrt \ ; \ (x,s) \mapsto \left((s\sqrt{\xi/\Psi(x)}+(1-s))x_{\lambda}\right)_{\lambda \in \str}.
\]
Thus, the inclusion $\iota: M\hookrightarrow\aGrt$ induces the isomorphism:
\begin{align}
\label{key prop 50}
\iota^*: H_{\KR}^*(M) \longrightarrow H_{\KC}^*(\aGrt).
\end{align}
Passing to the quotients, we obtain an equivariant map $\overline{\iota}:M/\wDR\to \wGr$ with respect to the inclusion $\wRR\hookrightarrow\wR$. This map can be shown to be a homeomorphism by a direct computation (See also \cite[Theorem 7.4]{Kirwan}). Hence, we obtain the isomorphism:
\begin{align}
\label{key prop 70}
\overline\iota^*:  H_{\wRR}^*(M/\wDR) \longrightarrow H_{\wR}^*(\wGr).
\end{align}
Let $\theta: ET \times_T M \to E\wR\times_{\wR} M/\wDR$ be a map induced by the quotient maps $M\to M/\wDR$ and $\KR\to\wRR$. Then we have the following commutative diagram.
\begin{equation}\label{complex to real}
\xymatrix{
H_{\wR}^*(\wGr)\ar[r]_{\whmap^*}\ar[d]_{\cong}^{\overline{\iota}^*} &H_{\KC}^*(\aGrt)\ar[d]_{\cong}^{\iota^*}\\
H_{\wRR}^*(M/\wDR)\ar[r]_{\theta^*} &H_{\KR}^*(M)
}
\end{equation}
Thus $\whmap^*$ is an isomorphism if $\theta^*$ is an isomorphism, which follows from Lemma \ref{key lemma}. \qed

\section{GKM Descriptions and Schubert Classes} \label{section-GKM}
In this section, we study the combinatorial presentations of $H_{\wR}^*(\wGr)$ and $H_{\KC}^*(\aGrt)$,  known as the \emph{GKM theory} developed in \cite{CS} and \cite{GKM}. This allows us, in particular, to show that the equivariant weighted Schubert classes $\wS_{\lambda}, \lambda \in\str$ form an $H^*(B\wR)$-module basis of $H_{\wR}^*(\wGr)$.

Recall that $H^*(B\KC)$ can be canonically identified with the symmetric algebra $\text{Sym}(\Lie(\KC)^*_{\Z} \otimes \Q)$ where $\Lie(\KC)^*_{\Z}$ is the space of $\Z$-linear functions on the integral lattice $\Lie(\KC)_{\Z} \subset \Lie(\KC)$. Since $\KC=(\C^{\times})^n\times\C^{\times}$ is a standard torus, we can take the standard $\Z$-basis $\{y_1,\cdots, y_n,\zz\}$ of $\Lie(\KC)^*_{\Z}$. 
Hence we let
\[
\Q[\KC^*] := H^*(B\KC) = \text{Sym}(\Lie(\KC)^*_{\Z} \otimes \Q) = \Q[y_1,\dots, y_n,\zz].
\]
Since $\wR$ is a quotient of $\KC$, we identify $\Lie(\wR)_{\Z}^*$ with its image in $\Lie(\KC)_{\Z}^*$ and it is easy to see that 
\[
\wy_i:=y_i-\frac{w_i}{a}\zz,  \quad  i=1,\cdots,n,
\]
form a basis of $\Lie(\wR)^*_{\Z} \otimes \Q$. We let
\begin{align*}
&\Q[\trR^*] := H^*(B\trR) = \text{Sym}(\Lie(\trR)^*_{\Z} \otimes \Q) = \Q[y_1,\dots, y_n] \quad \subset \Q[\KC^*], \\
&\Q[\wR^{\!\!*}] := H^*(B\wR) = \text{Sym}(\Lie(\wR)^*_{\Z} \otimes \Q) = \Q[\wy_1,\dots, \wy_n] \quad \subset \Q[\KC^*].
\end{align*}
We use the following notation in the rest of the paper:  for each $\lambda\in\str$, let
\begin{align}\label{new notations 1}
&y_{\lambda}:= \sum_{i \in \lambda} y_i \in \Q[\trR^*]   \ \ \ \ \  \mbox{ and } \ \ \ \ \ \  \  \wy_{\lambda}:= \sum_{i \in \lambda} \wy_i \in \Q[\wR^{\!\!*}].
\end{align}

The $\trR$-fixed points in $\Gr$ are the points $[e_{\mu}], \mu \in \str$ and the cohomology $H_{\trR}^*([e_{\mu}])$ is identified with $\Q[\trR^*]$. The restriction map to the fixed points
\begin{align}  \label{Gr loc}
H_{\trR}^*(\Gr) \to \bigoplus_{\mu \in \str} \Q[\trR^*]; \ \ \ \ \gamma \mapsto (\gamma|_{\mu})_{\mu \in \str}
\end{align}
is injective and the image is given by (see \cite{KT})
\begin{align} \label{Gr GKM}  
\left.\left\{\alpha = (\alpha(\mu))_{\mu} \in \bigoplus_{\mu\in\str} \Q[\trR^*] \,\right|
 \begin{matrix}
 \alpha(\lambda)-\alpha(\mu) \text{ is divisible by } y_{\lambda}-y_{\mu} \\ \text{ for any $\lambda$ and $\mu$ such that $|\lambda\cap\mu|=d-1$}
 \end{matrix}\right\}.
\end{align}
The fixed points of the $\wR$-action on $\wGr$ are again the images of $e_{\mu}$ in $\wGr$ and we also denote it by $[e_{\mu}]$. By identifying $H_{\wR}^*([e_{\mu}]) \cong \Q[\wR^{\!\!*}]$,   we have the restriction map
\begin{align}  \label{wGr loc}
H_{\wR}^*(\wGr) \longrightarrow \bigoplus_{\mu \in\str} \Q[\wR^{\!\!*}]; \ \ \ \ \gamma \mapsto (\gamma|_{\mu})_{\mu \in \str}.
\end{align}
For $\aGrt$, we restrict $H_{\KC}(\aGrt)$ to the complex $1$-dimensional orbits of $\KC $, which are given by $\Ctimes e_{\mu}$.  The isotropy subgroup $\KC_{\mu}$ at $e_{\mu}$ of the $\KC$-action is the kernel of the map $\KC \to \C^{\times}$ sending $(t_1,\cdots,t_n,s)$ to $s\cdot t_{\mu}$. It is connected and the inclusion $\KC_{\mu} \to \KC$ induces the isomorphism $\Lie(\KC_{\mu})^*_{\Z} \cong \Lie(\KC)^*_{\Z}/(y_\mu+\zz)$. Thus
\[
H^*_{\KC}(\Ctimes e_{\mu}) \cong H^*_{\KC_{\mu}}(e_{\mu}) \cong \Q[\KC_{\mu}^*] \cong \Q[\KC^*] / (y_{\mu}+\zz)
\]
and the restriction map is 
\begin{align} \label{aGr loc}
H_{\KC}^*(\aGrt) \longrightarrow \bigoplus_{\mu \in \str} \Q[\KC_{\mu}^*], \ \ \ \ P \mapsto (P|_{\mu})_{\mu \in \str}.
\end{align}
Putting (\ref{Gr loc}, \ref{wGr loc}, \ref{aGr loc}) together with $\hmap^*$ and $\whmap^*$, we have  the following commutative diagram
\begin{equation}\label{total-GKM}
\xymatrix{
H_{\trR}^*(\Gr) \ar[rr]\ar[d]_{\hmap^*}^{\cong}            && \bigoplus_{\mu} \Q[\trR^*] \ar[d]^{\cong}_{\kmap^*}\\
H_{\KC}^*(\aGrt) \ar[rr]                                    && \bigoplus_{\mu} \Q[\KC^*_{\mu}]\\
H_{\wR}^*(\wGr) \ar[rr]\ar[u]^{\whmap^*}_{\cong} && \bigoplus_{\mu} \Q[\wR^{\!\!*}]\ar[u]_{\cong}^{\wkmap^*}
}
\end{equation}
where the right vertical maps are induced from $\kmap_{\mu}:\KC_{\mu} \to \KC \to \trR$ and ${\wkmap}_{\!\mu}:\KC_{\mu} \to \KC \to \wR$ and they are isomorphisms because $\kmap_{\mu}$ and ${\wkmap}_{\!\mu}$ have finite kernels. 
The following are obtained by translating (\ref{Gr GKM}) to $H_{\KC}^*(\aGrt)$ and $H_{\wR}^*(\wGr)$ via this diagram. 
\begin{proposition}[GKM for $\wGr$] \label{GKM}
The restriction map (\ref{wGr loc}) is injective and the image is given by
\begin{align*}
 \left.\left\{\alpha \in \bigoplus_{\mu\in\str} \Q[\wR^{\!\!*}] \,\right|
 \begin{matrix}
 \alpha(\lambda)-\alpha(\mu) \emph{ is divisible by } w_{\mu}\wy_{\lambda}-w_{\lambda}\wy_{\mu} \\ \emph{ for any $\lambda$ and $\mu$ such that $|\lambda\cap\mu|=d-1$}
 \end{matrix}\right\}
\end{align*}
where $\wy_{\lambda}$ is defined in (\ref{new notations 1}).
\end{proposition}
\begin{proposition}[GKM for $\aGrt$]\label{aGr GKM}
The restriction map (\ref{aGr loc}) is injective and the image is given by
\begin{align*}
 \left.\left\{P \in \bigoplus_{\mu\in\str} \Q[\KC_{\mu}^*] \,\right|\,
 \begin{matrix}
 P(\lambda) = P(\mu) \ \ \ \mbox{  in }\  \Q[\KC^*]/(y_{\lambda}+\zz, y_{\mu}+\zz)
 \\ \emph{ for any $\lambda$ and $\mu$ such that $|\lambda\cap\mu|=d-1$}
 \end{matrix}\right\}.
\end{align*}
\end{proposition}
\noindent
\textit{Proof of Proposition \ref{aGr GKM} and Proposition \ref{GKM}}
\\
The injectivity of the maps (\ref{wGr loc}) and (\ref{aGr loc}) follows from the injectivity of the map (\ref{Gr loc}) by the commutativity of the diagram (\ref{total-GKM}). It remains to check that the GKM conditions are equivalent under the isomorphisms $\kmap^*$ and $\wkmap^*$. We prove it for $\wkmap$ because $\kmap$ is a special case of $\wkmap$. First note that, in Proposition \ref{GKM}, $\alpha(\lambda)-\alpha(\mu)$ is divisible by $w_{\mu}\wy_{\lambda}-w_{\lambda}\wy_{\mu}$ if and only if $ \alpha(\lambda)-\alpha(\mu) = 0$ in $\Q[\wR^{\!\!*}]/(w_{\mu}\wy_{\lambda}-w_{\lambda}\wy_{\mu})$. Therefore the GKM conditions are equivalent under $\wkmap^*$ if ${{\wkmap}_{\!\lambda}}^{\!\!*}$ and ${{\wkmap}_{\!\mu}}^{\!\!*}$ induce the isomorphism
\begin{align}\label{connection of GKM's 10}
\frac{\Q[\wR^{\!\!*}]}{(w_{\mu}\wy_{\lambda}-w_{\lambda}\wy_{\mu})} \to \frac{\Q[\KC^*]}{(y_{\lambda}+\zz, y_{\mu}+\zz)} , \ \ \ f \mapsto {{\wkmap}_{\!\lambda}}^{\!\!*}(f)={{\wkmap}_{\!\mu}}^{\!\!*}(f).
\end{align}
This follows from a straightforward computation. Indeed, we have
\begin{align}\label{connection of GKM's 20}
 w_{\lambda}\wy_{\mu} - w_{\mu}\wy_{\lambda}
 = w_{\lambda}(y_{\mu}+\zz)
 \quad \text{in } \Lie(\KC)^*_{\Q}/(y_\lambda+\zz)
\end{align}
and therefore the linear isomorphism ${\wkmap}_{\!\lambda}^*: \Lie(\wR)^*_{\Q} \to \Lie(\KC)^*_{\Q}/(y_\lambda+\zz)$ induces the linear isomorphism
\begin{align*}
\Lie(\wR)^*_{\Q}/(w_{\lambda}\wy_{\mu} - w_{\mu}\wy_{\lambda}) \cong \Lie(\KC)^*_{\Q}/(y_\lambda+\zz,y_{\mu}+\zz).
\end{align*}
\qed
\begin{remark}
Proposition \ref{aGr GKM} can be shown directly from Theorem 5.5 in \cite{HM} by using the description of $\wGr$ as the symplectic quotient of $\aGrt$ by the real torus $\wDR$ explained in Section \ref{subsection-Matsumura map}. 
\end{remark}
It is known that $\tilde{S}_{\lambda}|_{\lambda} = \prod_{(k,l)\in\invv(\lambda)} (y_{(k,l)\lambda} - y_{\lambda})$ and $\tilde{S}_{\lambda}|_{\mu}=0$ for all $\mu\ngeq\lambda$ (c.f. \cite{KT}). From this fact, together with the diagram (\ref{total-GKM}) and $\hmap^*(\tilde{S}_{\lambda})  = \aS_{\lambda}$, we have
\begin{proposition}\label{upper triangularity aS}
\begin{align*}
\aS_{\lambda}|_{\mu} = 
 \begin{cases}
  0 &\text{ if } \mu\ngeq\lambda, \\
  \prod_{(k,l)\in\invv(\lambda)} (y_{(k,l)\lambda}+\zz) &\text{ if } \mu=\lambda
 \end{cases}
 \ \ \ \ \ \ \ \mbox{ in } \Q[\KC^*]/(y_{\mu}+\zz).
\end{align*}
\end{proposition}
The next proposition is now immediate from Proposition \ref{upper triangularity aS}, the definition $\wS_{\lambda}=(\whmap^*)^{-1}(\aS_{\lambda})$ and (\ref{connection of GKM's 20}).
\begin{proposition}\label{can prop 50}
\begin{align*}
\emph{$\wS_{\lambda}$}|_{\mu} =
\begin{cases}
0 &  \text{ if $\mu\ngeq\lambda$},  \\
\prod_{(k,l)\in\invv(\lambda)}  \Big( \wy_{(k,l)\lambda} - \displaystyle{\frac{w_{(k,l)\lambda}}{w_{\lambda}}}\wy_{\lambda} \Big) &\text{ if $\mu=\lambda$}.
\end{cases}
\end{align*}
\end{proposition}

Having the upper-triangularity of the weighted Schubert classes as above, the proof of  \cite[Proposition 1]{KT} can be applied words by words to obtain
\begin{proposition} \label{can prop 120}
$\{\emph{$\wS_{\lambda}$}\}_{\lambda}$ is an $H^*(B\emph{$\wR$})$-module basis of $H_{\emph{$\wR$}}^*(\emph{$\wGr$})$.
\end{proposition}
\begin{example}
The followings is $\wS_{14}$ in $H_{\wR}^*(\text{wGr}(2,4))$: 
\[ 
\unitlength 0.1in
\begin{picture}( 47.4000, 17.9600)( 15.4000,-32.5600)
\put(55.7000,-19.0000){\makebox(0,0)[lb]{$ _{\{1,2\}}$}}%
\put(61.7200,-24.1000){\makebox(0,0)[lb]{$ _{\{1,4\}}$}}%
\put(49.7000,-24.2000){\makebox(0,0)[lb]{$ _{\{2,3\}}$}}%
\put(56.3000,-25.9000){\makebox(0,0)[lb]{$ _{\{2,4\}}$}}%
\put(57.7100,-22.6700){\makebox(0,0)[lb]{$ _{\{1,3\}}$}}%
\put(55.8000,-29.7000){\makebox(0,0)[lb]{$ _{\{3,4\}}$}}%
%
\special{pn 8}%
\special{pa 5698 1926}%
\special{pa 5258 2370}%
\special{fp}%
\special{pa 5258 2370}%
\special{pa 5698 2810}%
\special{fp}%
\special{pa 5698 2810}%
\special{pa 6140 2370}%
\special{fp}%
\special{pa 6140 2370}%
\special{pa 5698 1926}%
\special{fp}%
\special{pa 5258 2370}%
\special{pa 5588 2478}%
\special{fp}%
\special{pa 5588 2478}%
\special{pa 6140 2370}%
\special{fp}%
\special{pa 5698 1926}%
\special{pa 5588 2478}%
\special{fp}%
\special{pa 5588 2478}%
\special{pa 5698 2810}%
\special{fp}%
%
\special{pn 8}%
\special{pa 5698 1926}%
\special{pa 5754 2258}%
\special{dt 0.030}%
\special{pa 5754 2258}%
\special{pa 6140 2370}%
\special{dt 0.030}%
\special{pa 5754 2258}%
\special{pa 5258 2370}%
\special{dt 0.030}%
\special{pa 5754 2258}%
\special{pa 5698 2810}%
\special{dt 0.030}%
\put(19.5000,-17.1000){\makebox(0,0)[lb]{$(\tow_{23}-\frac{w_{23}}{w_{12}}\tow_{12})(\tow_{24}-\frac{w_{24}}{w_{12}}\tow_{12})$}}%
\put(30.9000,-24.6000){\makebox(0,0)[lb]{$(\tow_{24}-\frac{w_{24}}{w_{14}}\tow_{14})(\tow_{34}-\frac{w_{34}}{w_{14}}\tow_{14})$}}%
\put(16.4000,-24.5000){\makebox(0,0)[lb]{$0$}}%
\put(21.5000,-26.8000){\makebox(0,0)[lb]{$0$}}%
\put(25.1000,-22.3000){\makebox(0,0)[lb]{$(\tow_{23}-\frac{w_{23}}{w_{13}}\tow_{13})(\tow_{34}-\frac{w_{34}}{w_{13}}\tow_{13})$}}%
\put(23.6000,-31.9000){\makebox(0,0)[lb]{$0$}}%
%
\special{pn 8}%
\special{pa 2402 1744}%
\special{pa 1752 2396}%
\special{fp}%
\special{pa 1752 2396}%
\special{pa 2402 3044}%
\special{fp}%
\special{pa 2402 3044}%
\special{pa 3052 2396}%
\special{fp}%
\special{pa 3052 2396}%
\special{pa 2402 1744}%
\special{fp}%
\special{pa 1752 2396}%
\special{pa 2240 2558}%
\special{fp}%
\special{pa 2240 2558}%
\special{pa 3052 2396}%
\special{fp}%
\special{pa 2402 1744}%
\special{pa 2240 2558}%
\special{fp}%
\special{pa 2240 2558}%
\special{pa 2402 3044}%
\special{fp}%
%
\special{pn 8}%
\special{pa 2402 1744}%
\special{pa 2484 2232}%
\special{dt 0.030}%
\special{pa 2484 2232}%
\special{pa 3052 2396}%
\special{dt 0.030}%
\special{pa 2484 2232}%
\special{pa 1752 2396}%
\special{dt 0.030}%
\special{pa 2484 2232}%
\special{pa 2402 3044}%
\special{dt 0.030}%
%
\special{pn 8}%
\special{pa 1540 1460}%
\special{pa 6280 1460}%
\special{pa 6280 3256}%
\special{pa 1540 3256}%
\special{pa 1540 1460}%
\special{ip}%
\end{picture}%
 
\]
where the vertices are the elements of $\strEG$ and there is an edge for each pair of $\lambda$ and $\mu$ satisfying $|\lambda\cap\mu|=1$. 
\end{example}
\section{Structure Constants and Positivity} \label{section-final}
Since $\{\wS_{\lambda}\}_{\lambda}$ is an $H^*(B\wR)$-module basis of $H_{\wR}^*(\wGr)$, we can expand their pairwise cup product uniquely over $H^*(B\wR)$:
\begin{align}\label{equiv str const}
 \wS_{\lambda}\wS_{\mu} = \sum_{\nu} \wc_{\lambda\mu}^{\nu} \wS_{\nu} \ \ \  \mbox{where} \ \wc_{\lambda\mu}^{\nu} \in H^*(B\wR).
\end{align}
In \cite{GrahamPos} and \cite{KT}, it is shown that we can express $\trc_{\lambda,\mu}^{\nu}$ as a polynomial in $u_i$'s with non-negative coefficients where $u_i:=y_{i+1} - y_i\in\Lie(\trR)_{\Z}^*$ for each $i=1,\cdots,n-1$. In this section, we derive a formula for $\wc_{\lambda,\mu}^{\nu}$ from any given formula for $\trc_{\lambda,\mu}^{\nu}$. In particular, the formula of $\wc_{\lambda,\mu}^{\nu}$ is expressed in terms  an independent subset $\{\wu_i\}_{i=1,\cdots,n-1}$ of $\Lie(\wR)_{\Z}^*\otimes \Q$ in such a way that the positivity of $\trc_{\lambda,\mu}^{\nu}$ implies the positivity of $\wc_{\lambda,\mu}^{\nu}$ with respect to $\{\wu_i\}_{i=1,\cdots,n-1}$ if $w_1\leq\cdots \leq w_n$. Moreover, a manifestly positive formula for the structure constants $\{\w c_{\lambda\mu}^{\nu}\}$ of the ordinary cohomology $H^*(\wGr)$ is also obtained by specializing the one for $\wc_{\lambda\mu}^{\nu}$ at $\wu_1 = \cdots = \wu_{n-1}=0$.
\subsection{Equivariant Structure Constants}
We start with the following lemma which describes the divsor Schubert class.
\begin{lemma}\label{aSdiv=y_id+z}
Let $\id$ be the unique minimum in $\str$ with respect to the Bruhat order and $\di$  the unique element with $l(\id)=1$. We have $\aS_{\di}=(y_{\id}+\zz) \cdot 1$.
\end{lemma}
\begin{proof}
Since $\tS_{\di}|_{\mu} = y_{\id}-y_{\mu}$ (\cite[Lemma 3]{KT}) for each $\mu \in \str$, we have
\[
\aS_{\di}|_{\mu} = y_{\id}-y_{\mu} = y_{\id}+\zz \ \ \ \ \mbox{ in }  \Q[\KC^*]/(y_{\mu}+\zz).
\] 
Therefore the claim holds: $\aS_{\di}=(y_{\id}+z)\aS_{\id} = (y_{\id}+z)\cdot 1$. 
\end{proof}
Now we obtain the following expansion formula for the product $\aS_{\di} \aS_{\lambda}$ over $\Q[\wR^{\!\!*}]$.
\begin{proposition}\label{weighted Pieri-rule}
(\textit{The weighted Pieri-rule})
\[
\aS_{\di} \aS_{\lambda} = \Big(\wy_{\id} - \frac{w_{\id}}{w_{\lambda}} \wy_{\lambda}\Big) \aS_{\lambda} + \sum_{\lambda' \to \lambda} \frac{w_{\id}}{w_{\lambda}} \aS_{\lambda'}
\]
\end{proposition}
\begin{proof}
By the isomorphism $\hmap^*$, the equivariant Pieri-rule given in \cite[Proposition 2]{KT} implies
\begin{align}
\label{aff Pieri 100}
 \aS_{\di}\aS_{\lambda} = (y_{\id}-y_{\lambda})\aS_{\lambda}+\sum_{\lambda'\to\lambda}\aS_{\lambda'}
\end{align}
Together with Lemma \ref{aSdiv=y_id+z}, we obtain
\begin{align} \label{aff Pieri}
 0 = -(y_{\lambda}+\zz)\aS_{\lambda} +\sum_{\lambda'\to\lambda}\aS_{\lambda'}\ .
\end{align}
Multiply both sides of this equation by $\frac{w_{\id}}{w_{\lambda}}$, and then  again by Lemma \ref{aSdiv=y_id+z} we get
\begin{align*}
 \aS_{\di} \aS_{\lambda} = (y_{\id}+z)\aS_{\lambda}-\frac{w_{\id}}{w_{\lambda}}(y_{\lambda}+\zz)\aS_{\lambda} +\sum_{\lambda'\to\lambda}\frac{w_{\id}}{w_{\lambda}}\aS_{\lambda'}\ .
\end{align*}
Since $y_{\nu}+\zz = \wy_{\nu}+\frac{w_{\nu}}{a}\zz$ in $ \Lie(\KC)_{\Q}^*$ for all $\nu$, the terms with $z$ cancel and the claim follows.
\end{proof}
Let $K_{1^r \eta}^{\nu}$ be the coefficient in $\Q[\wR^{\!\!*}]$ for the following product.
\begin{align*}
(\aS_{\di})^{r} \aS_{\eta}
= \sum_{\nu} K_{1^r \eta}^{\nu} \aS_{\nu}.
\end{align*}
By applying Proposition \ref{weighted Pieri-rule} repeatedly, we can compute $K_{1^r \eta}^{\nu}$ explicitly. For example, for $r=2$, we have
\begin{align*}
(\aS_{\di})^2 \aS_{\eta}
&= \aS_{\di} \left( \Big(\wy_{\id} - \frac{w_{\id}}{w_{\eta}} \wy_{\eta}\Big) \aS_{\eta} + \sum_{\eta' \to \eta} \frac{w_{\id}}{w_{\eta}} \aS_{\eta'} \right) \\
&= \Big(\wy_{\id} - \frac{w_{\id}}{w_{\eta}} \wy_{\eta}\Big)^2 \aS_{\eta} \\
&\qquad\qquad
+ \sum_{\eta' \to \eta} \left( 
\Big(\wy_{\id} - \frac{w_{\id}}{w_{\eta'}} \wy_{\eta'}\Big) \frac{w_{\id}}{w_{\eta}}  
+ \frac{w_{\id}}{w_{\eta}} \Big(\wy_{\id} - \frac{w_{\id}}{w_{\eta}} \wy_{\eta}\Big) 
\right) \aS_{\eta'}
+ \sum_{\eta'' \to \eta' \to \eta} \frac{w_{\id}}{w_{\eta'}}\frac{w_{\id}}{w_{\eta}} \aS_{\eta''}.
\end{align*}
The general formula for $K_{1^r \eta}^{\nu}$ is recorded without proof as follows.
\begin{lemma}\label{lem:K1eta}
If $\nu \not\geq \mu$, $K_{1^r \eta}^{\nu}=0$.  If $\nu \geq \mu$, 
\begin{align}\label{weighted Kostka}
K_{1^r \eta}^{\nu} =\sum_{\substack{\nu = \nu^0 \rightarrow\nu^{1}\rightarrow\cdots \\ \rightarrow\nu^{l-1}\rightarrow \nu^l = \eta}} \sum_{J} \frac{w_{\nu}}{w_{\id}} \prod_{q=0}^{l} \frac{w_{\id}}{w_{\nu^q}}\Big(\wy_{\id} - \frac{w_{\id}}{w_{\nu^q}} \wy_{\nu^q}\Big)^{j_q},
\end{align}
where $l:=l(\nu)-l(\mu)$ and $J$ runs over all sequences $(j_0,\cdots,j_l)$ of non-negative integers satisfying $j_0+\cdots+j_l=r-l$. In particular, 
\[
K_{1^r \eta}^{\eta} =  \Big(\wy_{\id} - \frac{w_{\id}}{w_{\eta}} \wy_{\eta}\Big)^{r}.
\]
\end{lemma}

For each pair $\alpha=(i,j)$ of integers in $[n]$ such that $i>j$, let
\begin{align}\notag
 &u_{\alpha}:=y_i-y_j \in \Q[\trR^*], \\ \label{def of wu_alpha}
 &\wu_{\alpha}:=(\wy_i-\wy_j)-\frac{w_i-w_j}{w_{\id}}\wy_{\id} \in \Q[\wR^{\!\!*}], \\ \notag
 &w(\alpha):=w_i-w_j \ \ \in \Q. 
\end{align}
For simplicity, let $u_i:=u_{(i+1,i)}$ and $\wu_i:=\wu_{(i+1,i)}$ for $i=1,\cdots,n-1$. We can easily check that $\{\wu_1,\cdots,\wu_{n-1}\}$ is linearly independent in $\Lie(\wR)^*_{\Z}\otimes\Q$ and each $\wu_{\alpha}$ is a linear combination of $\wu_i$'s with non-negative coefficients.

The next proposition gives the essential equation to relate the $\Q[\trR^*]$-action to the $\Q[\wR^{\!\!*}]$-action in $H_{\KC}^*(\aGrt)$ and it follows from Lemma \ref{aSdiv=y_id+z} immediately.
\begin{proposition}\label{translation formula} In $H_{\KC}^*(\aGrt)$, we have
\begin{align*}
 y_i \cdot 1 = \Big(\wy_i-\frac{w_i}{w_{\id}}\wy_{\id}\Big) \cdot 1  + \frac{w_i}{w_{\id}}\aS_{\di} \ \ \ \mbox{ and }  \ \ u_{\alpha} \cdot 1  = \wu_{\alpha} \cdot 1  + \frac{w(\alpha)}{w_{\id}}\aS_{\di}.
\end{align*}
\end{proposition}

Let $Q$ be a formal variable. For each finite collection $I=\{\alpha_1,\cdots,\alpha_p\}$ of pairs of integers in $[n]$ as above, define $\wu_{I}^{(0)}, \wu_{I}^{(1)},\cdots, \wu_{I}^{(p)} \in \Q[\wR^{\!\!*}]$ by  
\[
\Big(\wu_{\alpha_1} + \frac{w(\alpha_1)}{w_{\id}}Q\Big) \cdots \Big(\wu_{\alpha_p} + \frac{w(\alpha_p)}{w_{\id}}Q\Big)
 = \sum_{r=0}^p \wu_{I}^{(r)} Q^r
\]
Explicitly, we have
\begin{align*}
\wu_{I}^{(r)}
=\sum_{1\leq s_1<\cdots<s_r\leq p}  \frac{w(\alpha_{s_1})}{w_{\id}}\cdots \frac{w(\alpha_{s_r})}{w_{\id}} \frac{\wu_{\alpha_1}\cdots\wu_{\alpha_p}}{\wu_{\alpha_{s_1}}\cdots\wu_{\alpha_{s_r}}}.
\end{align*}
For example, 
\begin{align*}
\wu_{I}^{(0)}=\wu_{\alpha_1}\cdots\wu_{\alpha_p} \ \ \ \ \mbox{ and } \ \ \ \ \wu_{I}^{(p)}=\frac{w(\alpha_1)}{w_{\id}}\cdots\frac{w(\alpha_p)}{w_{\id}}.
\end{align*}
Also note that, if $w_1=\cdots=w_n=0$, then $\wu_{I}^{(r)}=0$ for $r\geq1$. In this case, we denote $u_I :=\wu_I^{(0)}=\prod_{\alpha \in I} u_{\alpha}$.

It is known that the equivariant Schubert structure constant $\tilde{c}_{\lambda\mu}^{\nu}$ for $H_{\trR}^*(\Gr)$ is an element of $\Z[u_1,\cdots,u_{n-1}]$ 
\[
 \tilde{c}_{\lambda\mu}^{\nu}=\sum_{|I| = l(\lambda) + l(\mu) -l(\nu) } c(\lambda,\mu,\nu;I) u_{I}, \ \  c(\lambda,\mu,\nu;I) \in \Z
\]
where $I$ runs over collections of pairs $(i,j)$ of integers in $[n]$ with $i>j$ as above. For example, Knutson-Tao (\cite{KT}) computed the number $c(\lambda,\mu,\nu;I)$ in terms of the equivariant puzzles: with their notations, we have
\begin{align*}
 c(\lambda,\mu,\nu;I) = 
| \{ \text{equivariant puzzles $P$} \mid \text{$\partial P=\Delta_{\lambda\mu}^{\nu}$ and $\text{wt}(P)=u_{I}$} \} |.
\end{align*}

Now we state the main theorem of this section.
\begin{theorem} \label{main theorem}
Let $\lambda,\mu,\nu\in\str $, then
\begin{align}\label{statement of main thm}
\wc_{\lambda\mu}^{\nu}=\sum_{\nu\geq\eta\geq\lambda,\mu}\sum_{I}\sum_{r=0}^{|I|} c(\lambda,\mu,\eta;I) K_{1^r \eta}^{\nu} \wu_{I}^{(r)}
\end{align}
where $I=\{\alpha_1,\cdots, \alpha_{|I|}\}$ runs over collections of pairs $(i,j)$ of integers in $[n]$ with $i>j$. 
\end{theorem}
\begin{proof}
Since the map $\hmap^*$ is an isomorphism of rings over $\Q[\trR^*]$,  we have $\aS_{\lambda}\aS_{\mu}=\sum_{\nu} \tilde{c}_{\lambda\mu}^{\nu}\aS_{\nu}$ in $H_{\KC}(\aGrt)$.  Lemma \ref{translation formula} allows us to write
\begin{align*}
\tilde{c}_{\lambda\mu}^{\eta} 
&= \sum_{I} c(\lambda,\mu,\eta;I) \Big(\wu_{\alpha_1}+\frac{w(\alpha_1)}{w_{\id}}\aS_{\di}\Big)\cdots \Big(\wu_{\alpha_p}+\frac{w(\alpha_p)}{w_{\id}}\aS_{\di}\Big)\\
&= \sum_{I}\sum_{r=0}^{p} c(\lambda,\mu,\eta;I) \wu_{I}^{(r)} (\aS_{\di})^{r}.
\end{align*}
where we denoted $p:=|I|$. Therefore, 
\begin{align*}
\aS_{\lambda}\aS_{\mu} 
&=\sum_{\eta\geq\lambda,\mu}\sum_{I}\sum_{r=0}^{|I|} c(\lambda,\mu,\eta;I) \wu_{I}^{(r)} (\aS_{\di})^{r} \aS_{\eta} \\
&=\sum_{\eta\geq\lambda,\mu}\sum_{I}\sum_{r=0}^{|I|}
c(\lambda,\mu,\eta;I) \wu_{I}^{(r)}
\sum_{\nu\geq\eta} K_{1^r \eta}^{\nu} \aS_{\nu} \\
&=\sum_{\nu}\left(\sum_{\nu\geq\eta\geq\lambda,\mu}\sum_{I}\sum_{r=0}^{|I|}
c(\lambda,\mu,\eta;I) \wu_{I}^{(r)} K_{1^r \eta}^{\nu}\right) \aS_{\nu}.
\end{align*}
Since the coefficients are in $\Q[\wR^{\!\!*}]$, this proves the  desired formula.
\end{proof}
\begin{remark}
From the equivariant weighted Pieri rule, we can derive a recursive formula for the structure constants $\wc_{\lambda\mu}^{\nu}$, in the exactly same way shown in \cite[Theorem 3]{KT}:
\begin{equation}\label{recursive}
 \left(\wS_{\di}|_{\nu} - \wS_{\di}|_{\lambda}  \right)\wc_{\lambda\mu}^{\nu} = \left( \sum_{\lambda' \to \lambda} \frac{w_{\id}}{w_{\lambda}} \wc_{\lambda'\mu}^{\nu} - \sum_{\nu \to \nu'} \frac{w_{\id}}{w_{\nu'}}\wc_{\lambda\mu}^{\nu'}\right).
\end{equation}
However this equation plays no role in the derivation of our main formula (\ref{statement of main thm}), while the recursive formula in \cite{KT} plays a crucial role in their process of obtaining the original puzzle formula for $\tilde{c}_{\lambda\mu}^{\nu}$.
\end{remark}
\subsection{Positivity}
The equivariant positivity of \cite{GrahamPos} guarantees that the structure constants $\tilde{c}_{\lambda\mu}^{\nu}$ for $\Gr$ are polynomials in $u_1,\cdots,u_{n-1}$ with non-negative coefficients. In analogy to this fact, we prove the following equivariant positivity theorem for $\wGr$.
\begin{theorem}\label{equivariant positivity}
If $w_1\leq w_2\leq\cdots\leq w_n$, then $\wc_{\lambda\mu}^{\nu}$ is a polynomial in $\wu_1,\cdots,\wu_{n-1}$ with non-negative coefficients. 
\end{theorem}
\begin{proof}
We look at Theorem \ref{main theorem}. By the assumption, it is clear that $\wu_{I}^{(r)}$ is a polynomial in $\wu_1,\cdots,\wu_{n-1}$ with non-negative coefficients. For the positivity of $K_{1^r \eta}^{\nu}$, it suffices to show that $\wy_{\id}-\frac{w_{\id}}{w_{\nu}}\wy_{\nu} $ is a polynomial in $\wu_1,\cdots,\wu_{n-1}$ with non-negative coefficients for all $\nu\not=\id$. There exists $a \in [n]$ such that $(a,a+1)\nu=\nu'$ and $l(\nu')=l(\nu)-1$. The straightforward computation shows
\[
\wy_{\id}-\frac{w_{\id}}{w_{\nu}}\wy_{\nu} = \frac{w_{\nu'}}{w_{\nu}}\left(\wy_{\id}-\frac{w_{\id}}{w_{\nu'}}\wy_{\nu'} \right) + \frac{w_{\id}}{w_{\nu}}\wu_a
\]
Therefore the claim follows by the induction on the length of $\nu$. 
\end{proof}
\begin{remark}
Our positivity theorem holds for all weighted Grassmannians in the following sense:
for a given $\wGr$ with the weight $w=(w_1,\cdots, w_n)$, we can always perform a permutation on the basis $\{e_1,\cdots, e_n\}$ of $\C^n$ so that the new order on the weight is non-decreasing. Then we can re-define the Schubert classes $\{\wS_{\lambda}\}_{\lambda}$ to make sure that the structure constants are positive. 
\end{remark}
\subsection{Structure Constants for Ordinary Cohomology}
For each $\lambda\in\str$, define $ \w S_{\lambda} := \zeta^* (\wS_{\lambda}) \in H^*(\wGr )$ where $\zeta^*$ is the pullback along an inclusion as a fiber $\zeta: \wGr \to \wR\times_{\wR}\wGr$. Equivalently, this can be defined by $\w S_{\lambda}:=(\whmap^*)^{-1}[\af\Omega_{\lambda}]_{\wD}$ under the map $\whmap^*:H_{\wD}^*(\aGrt)\rightarrow H^*(\wGr)$.
\begin{corollary} \label{cor of main}
Let $\lambda,\mu,\nu\in\str$.  The structure constant $\w c_{\lambda\mu}^{\nu}$ is given by
\begin{align*}
\w c_{\lambda\mu}^{\nu}  = \sum_{\nu\geq\eta\geq\lambda,\mu}\sum_{\substack{\nu = \nu^0 \rightarrow \\ \cdots\rightarrow \nu^l = \eta}}\frac{ \tilde{c}_{\lambda\mu}^{\eta}(u_i=w_{i+1}-w_i, i=1,\cdots,n-1)  }{w_{\nu^1} \cdots w_{\nu^l}}
\end{align*}
if $l(\lambda)+l(\mu) =l(\nu)$ and $\w c_{\lambda\mu}^{\nu}=0$ if otherwise. Furthermore, if $w_1\leq w_2\leq \cdots\leq w_n$, $ \w c_{\lambda\mu}^{\nu}$ is non-negative. 
\end{corollary}
\begin{proof}
After the evaluation, each term in (\ref{statement of main thm}) can survive only if $l: =l(\nu)-l(\eta)=r=|I|$:
\begin{eqnarray*}
\wc_{\lambda\mu}^{\nu}|_{\wu_1=\cdots,\wu_{n-1}=0} 
&=&\sum_{\nu\geq\eta\geq\lambda,\mu}\sum_{|I|=l}c(\lambda,\mu,\eta;I) K_{1^l \eta}^{\nu} \wu_{I}^{(l)}\\
&=&\sum_{\nu\geq\eta\geq\lambda,\mu}\sum_{|I|=l}\sum_{\substack{\nu = \nu^0 \rightarrow \\ \cdots \rightarrow \nu^l = \eta}} c(\lambda,\mu,\eta;I) \frac{w(\alpha_1)\cdots w(\alpha_l)}{w_{\nu^1} \cdots w_{\nu^l}}\\
&=&\sum_{\nu\geq\eta\geq\lambda,\mu}\sum_{\substack{\nu = \nu^0 \rightarrow \\ \cdots\rightarrow \nu^l = \eta}}\frac{ \tilde{c}_{\lambda\mu}^{\eta}(u_i=w_{i+1}-w_i, i=1,\cdots,n-1)  }{w_{\nu^1} \cdots w_{\nu^l}}\\
\end{eqnarray*}
The positivity is a direct consequence of the equivaraiant positivity (Theorem \ref{equivariant positivity}).
\end{proof}
\subsection{Examples}
\begin{example}[Weighted Projective Space $\mathrm{wGr}(1,n)$]  The weighted projective space $\C\P_b=\C\P_{b_1,\cdots,b_n}$ is the quotient of $\C^n\backslash \{0\}$ by one dimensional torus $D_b:=\left\{(s^{b_1},\dots, s^{b_n}) \ |\ s\in \C^{\times} \right\} \subset R:=(\C^{\times})^n$ where $R$ acts on $\C^n\backslash \{0\}$ in the standard way. Let $\bar R_b:=R/D_b$ and $\{z_1,\cdots, z_n\}$ the standard basis of $(\Lie R)_{\Z}^*$. Then it is well-known that the $\bar R_b$-equivariant cohomology of $\C\P_{b_1,\cdots,b_n}$ is the corresponding Stanley-Reisner ring by regarding $\C\P_{b_1,\cdots,b_n}$ as a toric variety (c.f. \cite{DavisJanuszkiewicz91}):
\[
H_{\bar R_b}^*(\C\P_b) = \frac{\Q[z_1,\cdots, z_n]}{(z_1\cdots z_n)}.
\]
In our notation, $\mathrm{wGr}(1,n) = \C\P_b$ where $b_i = w_i + a$ and $K$ is related to $R$ via the map 
\[
K \to R;  (t_1,\dots, t_n, s) \mapsto (st_1,\dots, st_n).
\]
Therefore we can identify $H_{\bar R_b}^*(\C\P_b)$ with a subring of $H_{\wR}^*(\mathrm{wGr}(1,n))$ by $z_i = y_i + z$. With this identification, the Schubert classes $\wS_{\lambda}$ are given by 
\begin{equation}\label{basisProj}
\wS_{\{n\}}=1,\ \  \wS_{\{n-1\}} = z_n , \ \ \cdots ,\ \ \wS_{\{k\}} = z_{k+1} \cdots z_n , \ \  \cdots\ , \ \ \wS_{\{1\}} = z_2 \cdots z_n.
\end{equation}
where the Bruhat order is $\{n\} \leq \cdots \leq\{1\}$. The equivariant weighted Pieri rule gives 
\[
\wS_{\{n-1\}}\cdot \wS_{\{k\}} = \left(z_n - \frac{b_n}{b_k}z_k\right) \wS_{\{k\}} + \frac{b_n}{b_k} \wS_{\{k-1\}}.
\]
This is actually obvious in the presentation $\Q[z_1,\cdots, z_n]/(z_1\cdots z_n)$.
\end{example}
\begin{example}[Relation to the work of Kawasaki \cite{Kawasaki}]\label{kawasakibasis} In this example, all cohomologies are over $\Z$-coefficients. The integral cohomology of the weighted projective space $H^*(\C\P_{b};\Z)$ is known to be a free $\Z$-module generated by the \emph{Kawasaki basis} \cite{Kawasaki}. Namely, the map 
\[
\zeta:\C^n\backslash \{0\} \to \C^n\backslash \{0\}, \ \ \ (x_1,\cdots, x_n) \mapsto (x_1^{b_1},\cdots, x_n^{b_n})
\]
induces a map $\bar\zeta: \C\P^{n-1} \to \C\P_b$ and the inclusion $\bar\zeta^*: H^*(\C\P_b;\Z) \to H^*(\C\P^{n-1};\Z)$. Following \cite{DavisJanuszkiewicz91}, we represent $H^*(\C\P^{n-1};\Z)$ as
\[
\frac{\Z[z_1,\cdots, z_n]}{(z_1\cdots z_n, u_1,\cdots,u_{n-1})}
\]
where $u_i := z_{i+1} - z_i$. Then after we identity $H^*(\C\P_b;\Z)$ with the image of $\bar\zeta^*$, it is a free $\Z$-module generated by 
\[
\gamma_1:=l_1^b=1,\  \  \gamma_2:=l_2^b z_n, \ \  \gamma_3:=l_3^b z_{n-1}z_n,\ \  \cdots,\ \  \gamma_n:=l_n^b z_2 \cdots z_n
\]
where
\[
l_k^b := \mbox{l.c.m. of } \left\{\frac{b_{i_1} \cdots b_{i_k}}{\gcd(b_{i_1},\cdots, b_{i_k})} \Big|  1\leq i_1 <\cdots < i_k \leq n \right\}.
\]

On the other hand, the cohomology $H_{D_b}^*(\C^n\backslash \{0\};\Z)$ is known to be
\[
\frac{\Z[z_1,\cdots, z_n]}{(z_1\cdots z_n, u_1^b, \cdots, u_{n-1}^b)}
\]
where $\{u_i^b\}$ is a $\Z$-basis of $\Lie(\bar R_b)_{\Z}^* \subset (\Lie R)_{\Z}^*$ (c.f. Example 6.1 \cite{LMM}). We can regard our \emph{non-equivariant} Schubert classes $\w S_{\{k\}}$ as the monomial $z_{k+1} \cdots z_n$ in $H_{D_b}^*(\C^n\backslash\{0\};\Z)$. Together with the homomorphism $R \to R; (s_1,\dots, s_n) \mapsto (s_1^{b_1}, \dots, s_n^{b_n})$, $\zeta$ induces a map 
\[
\omega:  \C\P^{n-1} \to ED_b \times_{D_b} (\C^n\backslash\{0\})
\]
and the pullback $\omega^*$ is given by
\[
\omega^*: \frac{\Z[z_1,\cdots, z_n]}{(z_1\cdots z_n, u_1,\cdots,u_{n-1})} \to \frac{\Z[z_1,\cdots, z_n]}{(z_1\cdots z_n, u_1^b, \cdots, u_{n-1}^b)}; \ \ \ z_i \mapsto b_i z_i.
\]
Since $\bar\zeta$ factors through $\omega$ and the projection $\pi: ED_b\times_{D_b} \C^n\backslash \{0\} \to \C\P_b$ and by the fact that $H_{D_b}^*(\C^n\backslash\{0\};\Z)$ has no $\Z$-torsions in the degrees between $0$ and $2(n-1)$ (see Theorem 4.2 \cite{HolmWeighted}), we can conclude that the pullbacks of the Kawasaki's basis along the projection $\pi$ are the following multiples of our Schubert classes:
\[
\pi^*(\gamma_1)= \af S_{\{n\}} \ \ \mbox{ and }\ \   \pi^*(\gamma_{k})=\frac{l_{k}^b}{b_{n-k+2}b_{n-k+3}\cdots b_n}\af S_{\{n-k+1\}}, \ \ \ k=2,\cdots, n.
\]
\end{example}
\begin{example}[$\mathrm{wGr}(2,4)$] Here we demonstrate the computation of the product $\wS_{23} \wS_{23}$. By the upper triangularity of the GKM description of $\wS_{23}$, the product must be written by
\[
\wS_{23} \wS_{23} = \wc_{23,23}^{23} \wS_{23} + \wc_{23,23}^{13} \wS_{13} + \wc_{23,23}^{12} \wS_{12}.
\]
We can compute these coefficients from the formula of the following product for ordinary Grassmannian
\[
\tS_{23}\tS_{23} = (y_4-y_2) (y_4 - y_3) \tS_{23} + (y_4 - y_3) \tS_{13} + \tS_{12}.
\]
That is, we have
\begin{align*}
&c(23,23,23;I)=
\begin{cases}
 1 \quad \text{if $I=\{(4,2), (4,3)\}$}, \\
 0 \quad \text{otherwise},
\end{cases}
\\
&c(23,23,13;I)=
\begin{cases}
 1 \quad \text{if $I=\{(4,3)\}$}, \\
 0 \quad \text{otherwise},
\end{cases}
\\
&c(23,23,12;I)=
\begin{cases}
 1 \quad \text{if $I=\emptyset$}, \\
 0 \quad \text{otherwise},
\end{cases}
\end{align*}
and other $c(23,23,\eta;I)$'s are zero.
Here are the computation:
\begin{eqnarray*}
\wc_{23,23}^{23}&=& \left(\wy_4-\wy_2 - \frac{w_4-w_2}{w_{\id}}\wy_{\id}\right)\left(\wy_4-\wy_3 - \frac{w_4-w_3}{w_{\id}}\wy_{\id}\right) \\
&& + \Big(\wy_{\id}-\frac{w_{\id}}{w_{23}}\wy_{23}\Big) \left( \frac{w_4-w_2}{w_{\id}}\Big(\wy_4-\wy_3 - \frac{w_4-w_3}{w_{\id}}\wy_{\id}\Big) + \frac{w_4-w_3}{w_{\id}}\Big(\wy_4-\wy_2 - \frac{w_4-w_2}{w_{\id}}\wy_{\id}\Big) \right) \\
&& + \Big(\wy_{\id}-\frac{w_{\id}}{w_{23}}\wy_{23}\Big)^2 \frac{w_4-w_2}{w_{\id}} \frac{w_4-w_3}{w_{\id}}\\
\wc_{23,23}^{13} &=& \wy_4-\wy_3 - \frac{w_4-w_3}{w_{\id}}\wy_{\id}
+ \Big(\wy_{\id}-\frac{w_{\id}}{w_{13}}\wy_{13}\Big)\frac{w_4-w_3}{w_{\id}}\\
&& + \frac{w_4-w_2}{w_{23}}\Big(\wy_4-\wy_3 - \frac{w_4-w_3}{w_{\id}}\wy_{\id}\Big)+\frac{w_4-w_3}{w_{23}}\Big(\wy_4-\wy_2 - \frac{w_4-w_2}{w_{\id}}\wy_{\id}\Big) \\
&& +  \frac{w_{\id}}{w_{23}}\left(\Big(\wy_{\id}-\frac{w_{\id}}{w_{13}}\wy_{13}\Big)+\Big(\wy_{\id}-\frac{w_{\id}}{w_{23}}\wy_{23}\Big)\right)\frac{w_4-w_2}{w_{\id}}\frac{w_4-w_3}{w_{\id}}\\
\wc_{23,23}^{12} &=&  1 + \frac{w_4-w_3}{w_{13}} + \frac{w_4-w_2}{w_{23}}\frac{w_4-w_3}{w_{13}}
\end{eqnarray*}
Similarly we can also work out
\[
\wS_{23}\wS_{14} = \wc_{23,14}^{13}\wS_{13} + \wc_{23,14}^{12}\wS_{12}
\]
from $\tS_{23}\tS_{14} = (y_4 - y_1) \tS_{13}$ where
\begin{eqnarray*}
\wc_{23,14}^{13} &=& \Big(\wy_4-\wy_1 - \frac{w_4-w_1}{w_{\id}}\wy_{\id} \Big) + \Big( \wy_{\id}-\frac{w_{\id}}{w_{13}}\wy_{13} \Big) \frac{w_4-w_1}{w_{\id}} \\
\wc_{23,14}^{12} &=& \frac{w_4-w_1}{w_{13}}.
\end{eqnarray*}
\end{example}

\section*{Funding}
This work was supported by JSPS Research Fellowships for Young Scientists to H.A; and the National Research Foundation of Korea (NRF) grants funded by the Korea government (MEST) [2012-0000795, 2011-0001181] to T.M. 

\vspace{0.15in}
\noindent\textbf{Acknowledgment.}
The first author  is particularly indebted to Takashi Otofuji for many valuable discussions and helpful suggestions. The second author is grateful to Allen Knutson for teaching him about the subject and many inspirational discussions. He also would like to express his gratitude to the Algebraic Structure and its Application Research Institute at KAIST for providing him an excellent research environment in 2011-2012. The authors also would like to thank the referee for a thorough reading that improved the paper greatly, as well as Mikiya Masuda, Shintaro Kuroki, Hiroaki Ishida and Yukiko Fukukawa for organizing of the international conference \emph{Toric Topology in Osaka 2011} which made our collaboration possible.


\vspace{0.2in}
\textit{Hiraku Abe,
Osaka City University Advanced Mathematical Institute, 
3-3-138 Sugimoto, Sumiyoshi-ku Osaka 558-8585 JAPAN}

\textit{E-mail address: } hirakuabe@globe.ocn.ne.jp 
\\

\textit{Tomoo Matsumura,
Department of Mathematical Sciences,  Algebraic Structure and its Applications Research Center, KAIST, 291 Daehak-ro Yuseong-gu, Daejeon 305-701, SOUTH KOREA}

\textit{E-mail address: } tomoomatsumura@kaist.ac.kr
\\

\end{document}